\UseRawInputEncoding
\documentclass[11pt]{amsart}
\usepackage{url}
\usepackage{amssymb}
\usepackage{verbatim}
\usepackage{hyperref}
\usepackage{color}

\begin{document}

\newtheorem{theorem}{Theorem}    
\newtheorem{proposition}[theorem]{Proposition}
\newtheorem{conjecture}[theorem]{Conjecture}
\def\theconjecture{\unskip}
\newtheorem{corollary}[theorem]{Corollary}
\newtheorem{lemma}[theorem]{Lemma}
\newtheorem{sublemma}[theorem]{Sublemma}
\newtheorem{observation}[theorem]{Observation}
\newtheorem{remark}[theorem]{Remark}
\newtheorem{definition}[theorem]{Definition}
\theoremstyle{definition}
\newtheorem{notation}[theorem]{Notation}
\newtheorem{question}[theorem]{Question}
\newtheorem{questions}[theorem]{Questions}
\newtheorem{example}[theorem]{Example}
\newtheorem{problem}[theorem]{Problem}
\newtheorem{exercise}[theorem]{Exercise}

\numberwithin{theorem}{section} \numberwithin{theorem}{section}
\numberwithin{equation}{section}

\def\earrow{{\mathbf e}}
\def\rarrow{{\mathbf r}}
\def\uarrow{{\mathbf u}}
\def\varrow{{\mathbf V}}
\def\tpar{T_{\rm par}}
\def\apar{A_{\rm par}}

\def\reals{{\mathbb R}}
\def\torus{{\mathbb T}}
\def\heis{{\mathbb H}}
\def\integers{{\mathbb Z}}
\def\naturals{{\mathbb N}}
\def\complex{{\mathbb C}\/}
\def\distance{\operatorname{distance}\,}
\def\support{\operatorname{support}\,}
\def\dist{\operatorname{dist}\,}
\def\Span{\operatorname{span}\,}
\def\degree{\operatorname{degree}\,}
\def\kernel{\operatorname{kernel}\,}
\def\dim{\operatorname{dim}\,}
\def\codim{\operatorname{codim}}
\def\trace{\operatorname{trace\,}}
\def\Span{\operatorname{span}\,}
\def\dimension{\operatorname{dimension}\,}
\def\codimension{\operatorname{codimension}\,}
\def\nullspace{\scriptk}
\def\kernel{\operatorname{Ker}}
\def\ZZ{ {\mathbb Z} }
\def\p{\partial}
\def\rp{{ ^{-1} }}
\def\Re{\operatorname{Re\,} }
\def\Im{\operatorname{Im\,} }
\def\ov{\overline}
\def\eps{\varepsilon}
\def\lt{L^2}
\def\diver{\operatorname{div}}
\def\curl{\operatorname{curl}}
\def\etta{\eta}
\newcommand{\norm}[1]{ \|  #1 \|}
\def\expect{\mathbb E}
\def\bull{$\bullet$\ }
\def\C{\mathbb{C}}
\def\R{\mathbb{R}}
\def\Rn{{\mathbb{R}^n}}
\def\Sn{{{S}^{n-1}}}
\def\M{\mathbb{M}}
\def\N{\mathbb{N}}
\def\Q{{\mathbb{Q}}}
\def\Z{\mathbb{Z}}
\def\F{\mathcal{F}}
\def\L{\mathcal{L}}
\def\S{\mathcal{S}}
\def\supp{\operatorname{supp}}
\def\dist{\operatorname{dist}}
\def\essi{\operatornamewithlimits{ess\,inf}}
\def\esss{\operatornamewithlimits{ess\,sup}}
\def\xone{x_1}
\def\xtwo{x_2}
\def\xq{x_2+x_1^2}
\newcommand{\abr}[1]{ \langle  #1 \rangle}

\newcommand{\Norm}[1]{ \left\|  #1 \right\| }
\newcommand{\set}[1]{ \left\{ #1 \right\} }
\def\one{\mathbf 1}
\def\whole{\mathbf V}
\newcommand{\modulo}[2]{[#1]_{#2}}

\def\scriptf{{\mathcal F}}
\def\scriptg{{\mathcal G}}
\def\scriptm{{\mathcal M}}
\def\scriptb{{\mathcal B}}
\def\scriptc{{\mathcal C}}
\def\scriptt{{\mathcal T}}
\def\scripti{{\mathcal I}}
\def\scripte{{\mathcal E}}
\def\scriptv{{\mathcal V}}
\def\scriptw{{\mathcal W}}
\def\scriptu{{\mathcal U}}
\def\scriptS{{\mathcal S}}
\def\scripta{{\mathcal A}}
\def\scriptr{{\mathcal R}}
\def\scripto{{\mathcal O}}
\def\scripth{{\mathcal H}}
\def\scriptd{{\mathcal D}}
\def\scriptl{{\mathcal L}}
\def\scriptn{{\mathcal N}}
\def\scriptp{{\mathcal P}}
\def\scriptk{{\mathcal K}}
\def\frakv{{\mathfrak V}}

%
\newtheorem*{remark0}{\indent\sc Remark}
%
\renewcommand{\proofname}{\indent\sc Proof.} 

\title[Boundedness of intrinsic square functions and commutators]
{Boundedness of intrinsic square functions and commutators on generalized central Morrey spaces}
\subjclass[2010]{Primary 42B20; Secondary 42B25.}

%
\keywords{intrinsic square functions, commutators, $\lambda$-central $BMO$ functions, generalized central Morrey spaces.}
\thanks{The second author is the corresponding author.}
\thanks{The research was supported by the NNSF of China (No. 12061069.)}

\author[W. Lu]{Wenna Lu}
\address{Wenna Lu:
College of Mathematics and System Sciences\\ Xinjiang University\\
\"{U}r\"{u}mqi  830046\\ People's Republic of China}
\email{luwnmath@126.com}

\author[J. Zhou]{Jiang Zhou}
\address{Jiang Zhou: College of Mathematics and System Sciences\\ Xinjiang University\\
\"{U}r\"{u}mqi  830046\\ People's Republic of China}
\email{zhoujiang@xju.edu.cn}

\maketitle
\begin{abstract}
In this paper, the authors establish the boundedness for a large class of intrinsic square functions $\mathcal{G}_{\alpha}$, $g_{\alpha}$, $g^{\ast}_{\tilde{\lambda},\alpha}$ and their commutators $[b,\mathcal{G}_{\alpha}]$, $[b,g_{\alpha}]$ and $[b,g^{\ast}_{\tilde{\lambda},\alpha}]$  generated with $\lambda$-central $BMO$ functions $b\in CBMO^{p,\lambda}(\mathbb{R}^{n})$ on generalized central Morrey spaces $\mathcal{B}^{q,\varphi}(\mathbb{R}^{n})$ for $1<q<\infty,0<\alpha\leq1$, respectively. All of the results are new even on the central Morrey spaces $\mathcal{B}^{q,\lambda}(\mathbb{R}^{n})$.
\end{abstract}

\section{Introduction}

  Morrey in 1938, to study the local behavior of solutions to second-order elliptic partial differential equations, introduced the Morrey space in  \cite{M1}. It is well know that the space plays a significant role in studying the regularity of solutions to partial differential equations. Since then, many scholars have also considered the mapping properties of some classical operators in harmonic analysis on Morrey space (see \cite{CF},\cite{S}). The study on the intrinsic square function characterizations of function spaces has attracted a lot of attentions in recent years. To be precise, Wilson \cite{W1} first introduced intrinsic square functions to settle a conjecture proposed by Fefferman and Stein \cite{FS} on the boundedness of the Lusin-area function on the weighted Lebesgue space. Meanwhile, Wilson in \cite{W1} proved such operators are bounded from $L^{p}(\mathbb{R}^{n})$ to itself and also extended the weighted case. In 2012, Wang \cite{Wang1} obtained the boundedness of intrinsic square functions and their commutators generated by $BMO(\mathbb{R}^n)$ functions on weighted Morrey spaces. Later on, Wu and Zheng \cite{WZ} generalized these consequences to the generalized Morrey spaces. For more rich achievements and further developments in this subject, we refer the readers to \cite{Wang2,Wang3}, etc.

\section{Some Definitions and Notations}

Now, let us first recall the definitions of the intrinsic square functions (see \cite{W1}).

For $0<\alpha\leq1$, let $\mathcal{C}_{\alpha}$ be the family of functions $\phi:\mathbb{R}^{n}\rightarrow\mathbb{R}$ such that $\phi$ has support containing in  $\{x\in\mathbb{R}^{n}: |x|\leq1\}$, $\int_{\mathbb{R}^{n}}\phi(x)dx=0$, and for all $x,x'\in\mathbb{R}^n$
 $$|\phi(x)-\phi(x')|\leq|x-x'|^{\alpha}.$$
 For $(y,t)\in\mathbb{R}^{n+1}_{+}$ and $f\in L^1_{loc}(\mathbb{R}^{n})$, set
 $$\mathcal{A}_{\alpha}f(t,y)=\sup\limits_{\phi\in\mathcal{C}_{\alpha}}|f*{\phi}_{t}(y)|,$$
where ${\phi}_{t}(x)=\frac{1}{t^n}\phi(\frac{y}{t})$.

 The varying-aperture intrinsic square (intrinsic Lusin) function is defined by
 $$\mathcal{G}_{\alpha,\beta}(f)(x)=\Big(\int\int_{\Gamma_{\beta(x)}}\big(\mathcal{A}_{\alpha}f(t,y)\big)^2\frac{dydt}{t^{n+1}}\Big)^{1/2},$$
where
$$\Gamma_{\beta(x)}=\{(y,t)\in\mathbb{R}^{n+1}_{+}:|x-y|<\beta t\}.$$
If $\beta\equiv1$, we denote $\mathcal{G}_{\alpha,1}(f)$ by $\mathcal{G}_{\alpha}(f)$. Moreover, for any $0<\alpha\leq1$ and $\beta\geq1$, there is a pointwise relation between the function $\mathcal{G}_{\alpha,\beta}(f)(x)$ and $\mathcal{G}_{\alpha}(f)(x)$ as follows:
$$\mathcal{G}_{\alpha,\beta}(f)(x)\leq{\beta}^{\frac{3n}{2}+\alpha}\mathcal{G}_{\alpha}(f)(x).$$

A classical example is the kernel function ${\phi}_{t}(x)=P_t(x)$(Poisson kernel), We know that the intrinsic square functions are independent of any particular kernels, and it dominates pointwise the classical Lusin area integral and its some corresponding real-variable generations. On the other hand, we should pay attention to the fact that the function $\mathcal{G}_{\alpha,\beta}(f)(x)$ depends on the kernels with uniform compact support.

The intrinsic Littlewood-Paley $g$-function and the intrinsic $g_{\tilde{\lambda}}^{*}$-function are defined respectively by
$$g_{\alpha}f(x)=\Big(\int_{0}^{\infty}\big(\mathcal{A}_{\alpha}f(t,x)\big)^2\frac{dt}{t}\Big)^{1/2},$$
and
$$g_{\tilde{\lambda},\alpha}^{*}f(x)=\Big(\int\int_{\mathbb{R}^{n+1}_{+}}\Big(\frac{t}{t+|x-y|}\Big)^{n\tilde{\lambda}}\big(\mathcal{A}_{\alpha}f(t,y)\big)^2\frac{dydt}{t^{n+1}}\Big)^{1/2}.$$


Let $b$ be a locally integrable function on $\mathbb{R}^{n}$, in this paper, we also consider the commutators generated by the function $b$ and
the above intrinsic square functions, which are defined respectively by the following expressions (see\cite{Wang1}):
$$\mathcal{A}_{\alpha,b}f(t,y)=\sup\limits_{\phi\in\mathcal{C}_{\alpha}}\Big|\int_{\mathbb{R}^{n}}(b(y)-b(z)){\phi}_{t}(y-z)f(z)dz\Big|,$$
$$[b,\mathcal{G}_{\alpha}](f)(x)=\Big(\int\int_{\Gamma_{\beta(x)}}\big(\mathcal{A}_{\alpha,b}f(t,y)\big)^2\frac{dydt}{t^{n+1}}\Big)^{1/2},$$
$$[b,g_{\alpha}]f(x)=\Big(\int_{0}^{\infty}\big(\mathcal{A}_{\alpha,b}f(t,y)\big)^2\frac{dt}{t}\Big)^{1/2},$$
and
$$[b,g_{\tilde{\lambda},\alpha}^{*}]f(x)=\Big(\int\int_{\mathbb{R}^{n+1}_{+}}\Big(\frac{t}{t+|x-y|}\Big)^{n\tilde{\lambda}}\big(\mathcal{A}_{\alpha,b}f(t,y)\big)^2\frac{dydt}{t^{n+1}}\Big)^{1/2}.$$

Wiener \cite{WN1,WN2} explored a way to describe the behavior of a function at the infinity by considering the appropriate weighted $L^q(\mathbb{R}^{n})$ spaces, and then, Beurling \cite{B} employed this idea and obtained a pair of dual Banach spaces $A^q$ and $B^{q'}$, where $1/q+1/q'=1$. Subsequently, Lu and Yang \cite{LY1,LY2} introduced some new homogeneous Hardy space $H\dot{A}_q$ related to the homogeneous spaces $\dot{A}_q$, and obtained that the dual space of $H\dot{A}_q$ was the central bounded mean osciliation space $CBMO^q(\mathbb{R}^{n})$, which satisfies the following condition:
$$\|f\|_{CBMO^q(\mathbb{R}^{n})}=\sup\limits_{r>0}\Big(\frac{1}{|B(0,r)|}\int_{B(0,r)}|f(x)-f_{B(0,r)}|^qdx\Big)^{1/q}<\infty,$$
here and in the sequel, for $r>0$, $B(0,r)$ denote by the open ball centered at $0$ of radius $r$, $|B(0,r)|$ the Lebesgue measure of the ball $B(0,r)$ and
$$f_{B(0,r)}=\frac{1}{|B(0,r)|}\int_{B(0,r)}f(x)dx.$$

In fact, the space $CBMO^q(\mathbb{R}^{n})$ can be regarded as a local version of the space of bounded mean oscillation space $BMO(\mathbb{R}^{n})$. However, their properties have quite different, for example, the famous John-Nirenberg inequality for $BMO(\mathbb{R}^{n})$ space no longer holds in the $CBMO^q(\mathbb{R}^{n})$ space. In addition, Alvarez, Guzm\'{a}n-Partida and Lakey \cite{A} pointed out that $BMO(\mathbb{R}^{n})$ is strictly included in $\cap_{q>1}CBMO^q$. Furthermore, to study the relationship between central $BMO(\mathbb{R}^{n})$ spaces and Morrey spaces, they introduced $\lambda$-central bounded mean oscillation spaces and $\lambda$-central Morrey spaces, respectively.
\medskip

\quad\hspace{-22pt}{\bf Definition 2.1}(\cite{A}) {\it Let $\lambda<1/n$ and $1<q<\infty$. The $\lambda$-central bounded mean oscillation spaces
$CBMO^{q,\lambda}(\mathbb{R}^n)$ is defined as
$$CBMO^{q,\lambda}(\mathbb{R}^n)=\{f\in L_{loc}^q(\mathbb{R}^n):\|f\|_{CBMO^{q,\lambda}(\mathbb{R}^n)}<\infty\},$$
where
$$\|f\|_{CBMO^{q,\lambda}(\mathbb{R}^n)}=\sup\limits_{r>0}\Big(\frac{1}{|B(0,r)|^{1+\lambda q}}\int_{B(0,r)}|f(x)-f_{B(0,r)}|^qdx\Big)^{1/q}<\infty.$$
}

\quad\hspace{-22pt}{\bf Remark 2.1} {\it If two functions which differ by a constant are regarded as a function in the space $CBMO^{q,\lambda}(\mathbb{R}^n)$, then $CBMO^{q,\lambda}(\mathbb{R}^n)$ becomes a Banach space. when $\lambda=0$, the space  $CBMO^{q,\lambda}(\mathbb{R}^n)$ reduces to the space $CBMO^{q}(\mathbb{R}^n)$.
}

\quad\hspace{-22pt}{\bf Definition 2.2}(\cite{A}) {\it Let $\lambda\in\mathbb{R}$ and $1<q<\infty$. The $\lambda$-central Morrey space $\mathcal{B}^{q,\lambda}(\mathbb{R}^n)$ is defined by
$$\mathcal{B}^{q,\lambda}(\mathbb{R}^n)=\big\{f\in L_{loc}^q(\mathbb{R}^n): \|f\|_{\mathcal{B}^{q,\lambda}(\mathbb{R}^n)}<\infty\big\},$$
where
$$\|f\|_{\mathcal{B}^{q,\lambda}(\mathbb{R}^n)}=\sup\limits_{r>0}\Big(\frac{1}{|B(0,r)|^{1+\lambda q}}\int_{B(0,r)}|f(x)|^qdx\Big)^{1/q}.$$
}

\quad\hspace{-22pt}{\bf Remark 2.2} {\it $\mathcal{B}^{q,\lambda}(\mathbb{R}^n)$ space is a Banach space continuously included in $CBMO^{q,\lambda}(\mathbb{R}^n)$ space. When $\lambda<-1/q$, $\mathcal{B}^{q,\lambda}(\mathbb{R}^n)$ space reduces to $\{0\}$, and $\mathcal{B}^{q,-1/q}(\mathbb{R}^n)=L^q(\mathbb{R}^n)$.
}

\quad\hspace{-22pt}{\bf Remark 2.3} {\it(i) If $\lambda_1<\lambda_2$, it follows from the property of monotone functions that ${CBMO}^{q,\lambda_1}(\mathbb{R}^n)\subset{CBMO}^{q,\lambda_2}(\mathbb{R}^n)$ and $\mathcal{B}^{q,\lambda_1}(\mathbb{R}^n)\subset\mathcal{B}^{q,\lambda_2}(\mathbb{R}^n)$ for $1<q<\infty$.
}

{\it(ii) If $1<q_1<q_2<\infty$, ${CBMO}^{q_2,\lambda}(\mathbb{R}^n)\subset{CBMO}^{q_1,\lambda}(\mathbb{R}^n)$ for $\lambda<1/n$, and $\mathcal{B}^{q_2,\lambda}(\mathbb{R}^n)\subset\mathcal{B}^{q_1,\lambda}(\mathbb{R}^n)$ for $\lambda\in\mathbb{R}$.
}

In \cite{M2}, Mizuhara introduced the generalized Morrey spaces, and established the boundedness of some classical operators. After that, many authors have considered the mapping properties of various variant operators and their commutators on the generalized Morrey spaces. Naturally, the generalized central Morrey spaces, as a special of the local Morrey spaces, are also very significant function spaces in the study of boundedness of related operators, see for example \cite{Fan,FLL,SX,WZ,YT}, and references therein.

Next, we recall the definition of the generalized central Morrey space.

\quad\hspace{-22pt}{\bf Definition 2.3}(\cite{Fan}) {\it Let $\varphi(r)$ be a positive measurable function on $\mathbb{R}^+$, $\lambda\in\mathbb{R}$ and $1<q<\infty$. The generalized central Morrey space $\mathcal{B}^{q,\varphi}(\mathbb{R}^n)$ is defined by
$$\mathcal{B}^{q,\varphi}(\mathbb{R}^n)=\big\{f\in L_{loc}^q(\mathbb{R}^n): \|f\|_{\mathcal{B}^{q,\varphi}(\mathbb{R}^n)}<\infty\big\},$$
where
$$\|f\|_{\mathcal{B}^{q,\varphi}(\mathbb{R}^n)}=\sup\limits_{r>0}\frac{1}{\varphi(r)}\Big(\frac{1}{|B(0,r)|}\int_{B(0,r)}|f(x)|^qdx\Big)^{1/q}.$$
}

\quad\hspace{-22pt}{\bf Remark 2.4} {\it If $1<q_1<q_2<\infty$, $\mathcal{B}^{q_2,\varphi}(\mathbb{R}^n)\subset\mathcal{B}^{q_1,\varphi}(\mathbb{R}^n)$.
Note that if we take $\varphi(r)=r^{n\lambda}$, then $\mathcal{B}^{q,\varphi}(\mathbb{R}^n)=\mathcal{B}^{q,\lambda}(\mathbb{R}^n)$.
}

The main purpose in this paper is to establish the boundedenss of the operators $\mathcal{G}_{\alpha}$, $g_{\tilde{\lambda},\alpha}^{*}$, $g_{\alpha}$ and their commutators generated with the $\lambda$-central bounded mean oscillation function $b\in CBMO^{q,\lambda}(\mathbb{R}^n)$ on the generalized central Morrey spaces, respectively.

The rest of this paper is organized as follows. In Section 3 we establish main results of this paper and give several auxiliary lemmas, these lemmas are the important ingredients of this paper. The detailed proofs of Theorems \ref{th3.1}-\ref{th3.4} are given in Section 4. We would like to remark that some arguments are taken from \cite{Fan,WZ}.

As a rule, we use the symbol $f\lesssim g$ to denote there exists a positive constant $C$ such that $f\leq Cg $, and the notation $f\thickapprox g$ means that there exist positive constants $C_1, C_2$ such that $C_1 g\leq f\leq C_2g $. For any set $E\in\mathbb{R}^n$, ${\chi}_{E}$ denotes its characteristic function and $E^{c}$ denote its complementary set.

\section{ Main results and key lemmas}

In section, we first formulate the main results of this paper, and then state several key lemmas needed to prove these results.
\begin{theorem} \label{th3.1}
Let $1<q<\infty$, $0<\alpha\leq1$ and the pair $(\varphi_1,\varphi_2)$ satisfy the  condition
$$\int_{r}^{\infty}\frac{ess\inf_{t<\tau<\infty}\varphi_1(\tau){\tau}^{n/q}}{{t^{n/q+1}}}dt\lesssim\varphi_2(r).$$
Then the operator $\mathcal{G}_{\alpha}$ is bounded from $\mathcal{B}^{q,\varphi_1}(\mathbb{R}^n)$ to $\mathcal{B}^{q,\varphi_2}(\mathbb{R}^n)$.
\end{theorem}

\begin{theorem} \label{th3.2}
Let $1<q<\infty$, $0<\alpha\leq1$ and the pair $(\varphi_1,\varphi_2)$ satisfy the condition
$$\int_{r}^{\infty}\frac{ess\inf_{t<\tau<\infty}\varphi_1(\tau){\tau}^{n/q}}{{t^{n/q+1}}}dt\lesssim\varphi_2(r).$$
Then for $\tilde{\lambda}>3+\frac{2\alpha}{n}$, the operator $g_{\tilde{\lambda},\alpha}^{*}$ is bounded from $\mathcal{B}^{q,\varphi_1}(\mathbb{R}^n)$ to $\mathcal{B}^{q,\varphi_2}(\mathbb{R}^n)$.
\end{theorem}

\begin{theorem} \label{th3.3}
Let $1<q<\infty$, $0<\alpha\leq1$, $b\in{CBMO}^{p,\lambda}(\mathbb{R}^n)$, $0<\lambda<1/n$, $\frac{1}{q}=\frac{1}{q_1}+\frac{1}{p}$ and the pair $(\varphi_1,\varphi_2)$ satisfy the condition
$$\int_{r}^{\infty}\frac{ess\inf_{t<\tau<\infty}\varphi_1(\tau){\tau}^{n/q}}{{t^{n/q+1}}}dt\lesssim\varphi_2(r).$$
Then we have
$$\|[b,\mathcal{G}_{\alpha}](f)\|_{\mathcal{B}^{q,\varphi_2}(\mathbb{R}^n)}\lesssim\|b\|_{{CBMO}^{p,\lambda}(\mathbb{R}^n)}\|f\|_{\mathcal{B}^{q_1,\varphi_1}(\mathbb{R}^n)}.$$
\end{theorem}

\begin{theorem} \label{th3.4}
Let $1<q<\infty$, $0<\alpha\leq1$, $b\in{CBMO}^{p,\lambda}(\mathbb{R}^n)$, $0<\lambda<1/n$, $\frac{1}{q}=\frac{1}{q_1}+\frac{1}{p}$ and the pair $(\varphi_1,\varphi_2)$ satisfy the condition
$$\int_{r}^{\infty}\frac{ess\inf_{t<\tau<\infty}\varphi_1(\tau){\tau}^{n/q}}{{t^{n/q+1}}}dt\lesssim\varphi_2(r).$$
Then for $\tilde{\lambda}>3+\frac{2\alpha}{n}$, we have
$$\|[b,g_{\tilde{\lambda},\alpha}^{*}](f)\|_{\mathcal{B}^{q,\varphi_2}(\mathbb{R}^n)}\lesssim\|b\|_{{CBMO}^{p,\lambda}(\mathbb{R}^n)}\|f\|_{\mathcal{B}^{q_1,\varphi_1}(\mathbb{R}^n)}.$$
\end{theorem}
In addition, the author in \cite{W1} proved the operators $\mathcal{G}_{\alpha}$ and $g_{\alpha}$ are pointwise comparable. Therefore, as applications of Theorems \ref{th3.1} and \ref{th3.3}, we have the following conclusions.

\quad\hspace{-22pt}{\bf Corollary 3.1.}{\it \;Let $1<q<\infty$, $0<\alpha\leq1$ and the pair $(\varphi_1,\varphi_2)$ satisfy the condition
$$\int_{r}^{\infty}\frac{ess\inf_{t<\tau<\infty}\varphi_1(\tau){\tau}^{n/q}}{{t^{n/q+1}}}dt\lesssim\varphi_2(r).$$
Then the operator $g_{\alpha}$ is bounded from $\mathcal{B}^{q,\varphi_1}(\mathbb{R}^n)$ to $\mathcal{B}^{q,\varphi_2}(\mathbb{R}^n)$.}

\quad\hspace{-22pt}{\bf Corollary 3.2.}{\it \;
Let $1<q<\infty$, $0<\alpha\leq1$, $b\in{CBMO}^{p,\lambda}(\mathbb{R}^n)$, $0<\lambda<1/n$, $\frac{1}{q}=\frac{1}{q_1}+\frac{1}{p}$ and the pair $(\varphi_1,\varphi_2)$ satisfy the condition
$$\int_{r}^{\infty}\frac{ess\inf_{t<\tau<\infty}\varphi_1(\tau){\tau}^{n/q}}{{t^{n/q+1}}}dt\lesssim\varphi_2(r).$$
Then we have
$$\|[b,g_{\alpha}](f)\|_{\mathcal{B}^{q,\varphi_2}(\mathbb{R}^n)}\lesssim\|b\|_{{CBMO}^{p,\lambda}(\mathbb{R}^n)}\|f\|_{\mathcal{B}^{q_1,\varphi_1}(\mathbb{R}^n)}.$$}

\quad\hspace{-22pt}{\bf Remark 3.1} {\it According to Remark 2,4, one can see that all of the results are new on the central Morrey spaces $\mathcal{B}^{q,\lambda}(\mathbb{R}^{n})$.
}

Next, we give some important lemmas, which will play significant roles in the process of proof.

\begin{lemma}\label{le3.1}
Let $1<q<\infty$ and $0<\alpha\leq1$, then the inequality
$$\|\mathcal{G}_{\alpha}(f)\|_{L^q(B(0,r))}\lesssim r^{n/q}\int_{2r}^{\infty}t^{-n/q-1}\|f\|_{L^q(B(0,t))}dt$$
holds for any ball $B(0,r)$ and for all $f\in L^q_{loc}(\mathbb{R}^n)$.
\end{lemma}

\begin{lemma}\label{le3.2}
Let $1<q<\infty$ and $0<\alpha\leq1$, then for any ball $B(0,r)$ and for all $f\in L^q_{loc}(\mathbb{R}^n)$, the operator
$$\mathcal{G}_{\alpha,2^{j}}(f)(x)=\Big(\int_{0}^{\infty}\int_{|x-y|\leq2^{j}t}\big(\mathcal{A}_{\alpha}f(t,y)\big)^2\frac{dydt}{t^{n+1}}\Big)^{1/2}, \quad j\in\mathbb{Z}^+$$
satisfies the following inequality
$$\|\mathcal{G}_{\alpha,2^{j}}(f)\|_{L^q(B(0,r))}\lesssim 2^{j(3+\frac{2\alpha}{n})}r^{n/q}\int_{2r}^{\infty}t^{-n/q-1}\|f\|_{L^q(B(0,t))}dt.$$
\end{lemma}

\begin{lemma}\label{le3.3}
Let $1<q<\infty$, $0<\alpha\leq1$, $b\in{CBMO}^{p,\lambda}(\mathbb{R}^n)$, $0<\lambda<1/n$, $\frac{1}{q}=\frac{1}{q_1}+\frac{1}{p}$, then the inequality
$$\|[b,\mathcal{G}_{\alpha}](f)\|_{L^q(B(0,r))}\lesssim r^{n/q}\|b\|_{{CBMO}^{p,\lambda}(\mathbb{R}^n)}\int_{2r}^{\infty}t^{n\lambda}\big(1+\ln\frac{t}{r}\big)\|f\|_{L^{q_1}(B(0,t))}\frac{dt}{t^{n/q_1+1}}$$
holds for any ball $B(0,r)$ and for all $f\in L^{q_1}_{loc}(\mathbb{R}^n)$.
\end{lemma}

\section{Proofs of main results }

In section, we first establish Lemmas \ref{le3.1}-\ref{le3.3}.

\begin{proof}[Proof of Lemma \ref{le3.1}]
For any $r>0$, set $B=B(0,r)$ and $2B=B(0,2r)$, we write
$$f(x)=f(x){\chi}_{2B}(x)+f(x){\chi}_{(2B)^{c}}(x):=f_1(x)+f_2(x)$$
and have
\begin{align*}
\|\mathcal{G}_{\alpha}(f)\|_{L^q(B)}&\leq\|\mathcal{G}_{\alpha}(f_1)\|_{L^q(B)}+\|\mathcal{G}_{\alpha}(f_2)\|_{L^q(B)}\\
&:=I_1+I_2.
\end{align*}
For $I_1$, since $f_1\in L^q(\mathbb{R}^n)$, $\mathcal{G}_{\alpha}(f_1)\in{L^q(\mathbb{R}^n)}$, by using the boundedness of $\mathcal{G}_{\alpha}$ on $L^q(\mathbb{R}^n)$ in  \cite{W1},  we get
\begin{align*}
I_1=\|\mathcal{G}_{\alpha}(f_1)\|_{L^q(B)}\leq\|\mathcal{G}_{\alpha}(f_1)\|_{L^q(\mathbb{R}^n)}&\lesssim\|f_1\|_{L^q(\mathbb{R}^n)}\\
&\lesssim r^{n/q}\int_{2r}^{\infty}t^{-n/q-1}\|f\|_{L^q(B(0,t))}dt.
\end{align*}
We now estimate $I_2$, note that the fact $\|\phi\|_{L^{\infty}}\lesssim1$, we obtain that
  $$|f_2*{\phi}_{t}(y)|=\Big|\frac{1}{t^{n}}\int_{|y-z|\leq t}\phi(\frac{y-z}{t})f_{2}(z)dz\Big|\lesssim\frac{1}{t^{n}}\int_{|y-z|\leq t}|f_{2}(z)|dz.$$
It is clear that $x\in B$, $(y,t)\in \Gamma(x)$, and $z\in{(2B)^{c}}$, which deduce that
$$r\leq|z|-|x|\leq|x-z|\leq|x-y|+|y-z|\leq2t.$$
From this, we get
\begin{align*}
\mathcal{G}_{\alpha}(f_2)(x)&\lesssim\Big(\int\int_{\Gamma(x)}\Big(\frac{1}{t^{n}}\int_{|y-z|\leq t}|f_{2}(z)|dz\Big)^2\frac{dydt}{t^{n+1}}\Big)^{1/2}\\
&\leq\Big(\int_{t>\frac{r}{2}}\int_{|x-y|<t}\Big(\int_{|z-x|\leq2t}|f_{2}(z)|dz\Big)^2\frac{dydt}{t^{3n+1}}\Big)^{1/2}\\
&\lesssim\Big(\int_{t>\frac{r}{2}}\Big(\int_{|z-x|\leq2t}|f_{2}(z)|dz\Big)^2\frac{dt}{t^{2n+1}}\Big)^{1/2}.
\end{align*}
The Fubini theorem and the fact that $|z-x|\geq|z|-|x|\geq\frac{1}{2}|z|$ lead to the following result
\begin{align*}
\mathcal{G}_{\alpha}(f_2)(x)&\leq\Big(\int_{\mathbb{R}^n}\Big(\int_{t>\frac{|z-x|}{2}}\frac{dt}{t^{2n+1}}\Big)^2\Big)^{1/2}|f_{2}(z)|dz\\
&\lesssim\int_{|z|>2r}\frac{|f(z)|}{|x-z|^n}dz\\
&\lesssim\int_{|z|>2r}\frac{|f(z)|}{|z|^n}dz  \quad\quad(|z|\thickapprox|x-z|)\\
&\lesssim\int_{|z|>2r}|f(z)|\int_{|z|}^{\infty}\frac{1}{t^{n+1}}dtdz\\
&\lesssim\int_{2r}^{\infty}\int_{2r<|z|<t}|f(z)|\frac{1}{t^{n+1}}dzdt\\
&\lesssim\int_{2r}^{\infty}t^{-n/q-1}\|f\|_{L^q(B(0,t))}dt.
\end{align*}
Thus, we have
$$I_2=\|\mathcal{G}_{\alpha}(f_2)\|_{L^q(B)}\lesssim r^{n/q}\int_{2r}^{\infty}t^{-n/q-1}\|f\|_{L^q(B(0,t))}dt.$$
Combining the estimates $I_1$ and $I_2$, the proof of Lemma \ref{le3.1} is completed.
\end{proof}

\begin{proof}[Proof of Lemma \ref{le3.2}]
For $0<\alpha\leq1$ and $\beta\geq1$, we know that
$$\mathcal{G}_{\alpha,\beta}(f)(x)\leq{\beta}^{3+\frac{2\alpha}{n}}\mathcal{G}_{\alpha}(f)(x).$$
Now, we set $\beta=2^{j}>1$, which, together with Lemma \ref{le3.1}, gives that the desired result.
\end{proof}

\begin{proof}[Proof of Lemma \ref{le3.3}]
Let $1<q<\infty$, $0<\alpha\leq1$, $b\in{CBMO}^{p,\lambda}(\mathbb{R}^n)$, $0<\lambda<1/n$, $\frac{1}{q}=\frac{1}{q_1}+\frac{1}{p}$.
For any $r>0$, set $B=B(0,r)$ and $2B=B(0,2r)$, we write
$$f(x)=f(x){\chi}_{2B}(x)+f(x){\chi}_{(2B)^{c}}(x):=f_1(x)+f_2(x),$$
and
\begin{align*}
[b,\mathcal{G}_{\alpha}](f)(x)&=[b-b_B,\mathcal{G}_{\alpha}](f)(x)\\
&\leq(b(x)-b_B)\mathcal{G}_{\alpha}(f_1)(x)+\mathcal{G}_{\alpha}((b-b_B)f_1)(x)\\
&\quad+(b(x)-b_B)\mathcal{G}_{\alpha}(f_2)(x)+\mathcal{G}_{\alpha}((b-b_B)f_2)(x).
\end{align*}
Hence, we have
\begin{align*}
&\|[b,\mathcal{G}_{\alpha}](f)\|_{L^q(B(0,r))}\\
&\leq\|(b(x)-b_B)\mathcal{G}_{\alpha}(f_1)\|_{L^q(B(0,r))}+\|\mathcal{G}_{\alpha}((b-b_B)f_1)\|_{L^q(B(0,r))}\\
&\quad+\|(b(x)-b_B)\mathcal{G}_{\alpha}(f_2)\|_{L^q(B(0,r))}+\|\mathcal{G}_{\alpha}((b-b_B)f_2)\|_{L^q(B(0,r))}\\
&:=J_1+J_2+J_3+J_4.
\end{align*}
For $J_1$, by H\"{o}lder's inequality and the boundedness of $\mathcal{G}_{\alpha}$, we get
\begin{align*}
J_1&=\Big(\int_{B}|b(x)-b_B|^q|\mathcal{G}_{\alpha}(f_1)(x)|^qdx\Big)^{1/q}\\
&\leq\Big(\int_{B}|b(x)-b_B|^{p}dx\Big)^{1/p}\Big(\int_{B}|\mathcal{G}_{\alpha}(f_1)(x)|^{q_1}dx\Big)^{1/q_1}\\
&\lesssim r^{n/p+n\lambda}\|b\|_{{CBMO}^{p,\lambda}(\mathbb{R}^n)}\|f\|_{L^{q_1}(2B)}\\
&\lesssim r^{n/p+n\lambda+n/q_1}\|b\|_{{CBMO}^{p,\lambda}(\mathbb{R}^n)}\int_{2r}^{\infty}\|f\|_{L^{q_1}(B(0,t))}\frac{dt}{t^{n/q_1+1}}\\
&\lesssim r^{n/q}\|b\|_{{CBMO}^{p,\lambda}(\mathbb{R}^n)}\int_{2r}^{\infty}t^{n\lambda}\|f\|_{L^{q_1}(B(0,t))}\frac{dt}{t^{n/q_1+1}}\\
&\lesssim r^{n/q}\|b\|_{{CBMO}^{p,\lambda}(\mathbb{R}^n)}\int_{2r}^{\infty}t^{n\lambda}\big(1+\ln\frac{t}{r}\big)\|f\|_{L^{q_1}(B(0,t))}\frac{dt}{t^{n/q_1+1}}.
\end{align*}
Similarly, for $J_2$, we also have
\begin{align*}
J_2&=\Big(\int_{B}|\mathcal{G}_{\alpha}((b-b_B)f_1)(x)|^qdx\Big)^{1/q}\\
&\lesssim\Big(\int_{2B}|b(x)-b_B|^{p}dx\Big)^{1/p}\Big(\int_{2B}|f(x)|^{q_1}dx\Big)^{1/q_1}\\
&\lesssim\Big(\int_{2B}|b(x)-b_{2B}|^{p}dx\Big)^{1/p}\Big(\int_{2B}|f(x)|^{q_1}dx\Big)^{1/q_1}\\
&\quad+|2B|^{1/p}|b_B-b_{2B}|\Big(\int_{2B}|f(x)|^{q_1}dx\Big)^{1/q_1}\\
&\lesssim\Big(\int_{2B}|b(x)-b_{2B}|^{p}dx\Big)^{1/p}\|f\|_{L^{q_1}(2B)}\\
&\lesssim r^{n/p+n\lambda}\|b\|_{{CBMO}^{p,\lambda}(\mathbb{R}^n)}\|f\|_{L^{q_1}(2B)}\\
&\lesssim r^{n/q}\|b\|_{{CBMO}^{p,\lambda}(\mathbb{R}^n)}\int_{2r}^{\infty}t^{n\lambda}\big(1+\ln\frac{t}{r}\big)\|f\|_{L^{q_1}(B(0,t))}\frac{dt}{t^{n/q_1+1}}.
\end{align*}
For $J_3$, by Lemma \ref{le3.1}, we know that
$$\mathcal{G}_{\alpha}(f_2)(x)\lesssim\int_{2r}^{\infty}t^{-n/{q_1}-1}\|f\|_{L^{q_1}(B(0,t))}dt,$$
which, together with H\"{o}lder's inequality, implies that
\begin{align*}
J_3&\lesssim\Big(\int_{B}|b(x)-b_B|^{q}dx\Big)^{1/q}\int_{2r}^{\infty}\|f\|_{L^{q_1}(B(0,t))}\frac{dt}{t^{n/q_1+1}}\\
&\leq r^{n/p+n\lambda+n/q_1}\|b\|_{{CBMO}^{p,\lambda}(\mathbb{R}^n)}\int_{2r}^{\infty}\|f\|_{L^{q_1}(B(0,t))}\frac{dt}{t^{n/q_1+1}}\\
&\lesssim r^{n/q}\|b\|_{{CBMO}^{p,\lambda}(\mathbb{R}^n)}\int_{2r}^{\infty}t^{n\lambda}\big(1+\ln\frac{t}{r}\big)\|f\|_{L^{q_1}(B(0,t))}\frac{dt}{t^{n/q_1+1}}.
\end{align*}
For $J_4$, since $|y-x|<t$, it follows that $|x-z|<2t$. By the Minkowski inequality, we have
\begin{align*}
\mathcal{G}_{\alpha}((b-b_B)f_2)(x)&\lesssim\Big(\int\int_{\Gamma(x)}\Big(\int_{|x-z|<2t}|b(z)-b_B||f_{2}(z)|dz\Big)^2\frac{dydt}{t^{3n+1}}\Big)^{1/2}\\
&\leq\Big(\int_{0}^{\infty}\Big(\int_{|x-z|<2t}|b(z)-b_B||f_{2}(z)|dz\Big)^2\frac{dt}{t^{2n+1}}\Big)^{1/2}\\
&\lesssim\int_{|z|>2r}|b(z)-b_B||f_{2}(z)|\frac{1}{|x-z|^n}dz.
\end{align*}
Note that $|z-x|\geq\frac{1}{2}|z|$, by applying the Fubini theorem, we get
\begin{align*}
J_4&\lesssim\Big(\int_{B(0,r)}\Big|\int_{|z|>2r}|b(z)-b_B||f_{2}(z)|\frac{1}{|x-z|^n}dz\Big|^{q}dx\Big)^{1/q}\\
&\lesssim r^{n/q}\int_{|z|>2r}|b(z)-b_B||f_{2}(z)|\frac{1}{|z|^n}dz\\
&\lesssim r^{n/q}\int_{|z|>2r}|b(z)-b_B||f(z)|\int_{|z|}^{\infty}\frac{1}{t^{n+1}}dtdz\\
&\leq r^{n/q}\int_{2r}^{\infty}\int_{B(0,t)}|b(z)-b_B||f(z)|dz\frac{1}{t^{n+1}}dt\\
&\leq r^{n/q}\int_{2r}^{\infty}\int_{B(0,t)}|b(z)-b_{B(0,t)}||f(z)|dz\frac{1}{t^{n+1}}dt\\
&\quad + r^{n/q}\int_{2r}^{\infty}\int_{B(0,t)}|b_{B(0,t)}-b_B||f(z)|dz\frac{1}{t^{n+1}}dt\\
&:=J_{41}+J_{42}.
\end{align*}

For $J_{41}$, by the H\"{o}lder inequality, we have
\begin{align*}
&\int_{B(0,t)}|b(z)-b_{B(0,t)}||f(z)|dz\\
&\lesssim t^{n(1-1/q)}\Big(\int_{B(0,t)}|b(z)-b_{B(0,t)}|^q|f(z)|^qdz\Big)^{1/q}\\
&\lesssim t^{n(1-1/q)+n/p+n\lambda}\|b\|_{{CBMO}^{p,\lambda}(\mathbb{R}^n)}\|f\|_{L^{q_1}(B(0,t))}.
\end{align*}
Thus,
\begin{align*}
J_{41}&\leq r^{n/q}\|b\|_{{CBMO}^{p,\lambda}(\mathbb{R}^n)}\int_{2r}^{\infty}t^{n(1-1/q)+n/p+n\lambda}\|f\|_{L^{q_1}(B(0,t))}\frac{1}{t^{n+1}}dt\\
&\lesssim r^{n/q}\|b\|_{{CBMO}^{p,\lambda}(\mathbb{R}^n)}\int_{2r}^{\infty}t^{n\lambda}\big(1+\ln\frac{t}{r}\big)\|f\|_{L^{q_1}(B(0,t))}\frac{dt}{t^{n/q_1+1}}.
\end{align*}
For $J_{42}$, we also have
\begin{align*}
J_{42}&=r^{n/q}\int_{2r}^{\infty}|b_{B(0,t)}-b_B|\int_{B(0,t)}|f(z)|dz\frac{1}{t^{n+1}}dt\\
&\leq r^{n/q}\int_{2r}^{\infty}\Big(\frac{1}{|B(0,t)|}\int_{B(0,t)}|b(x)-b_{B(0,r)}|^pdx\Big)^{1/p}\int_{B(0,t)}|f(z)|dz\frac{1}{t^{n+1}}dt\\
&\leq r^{n/q}\|b\|_{{CBMO}^{p,\lambda}(\mathbb{R}^n)}\int_{2r}^{\infty}t^{n\lambda}\big(1+\ln\frac{t}{r}\big)\|f\|_{L^{q_1}(B(0,t))}t^{n(1-1/q_1)}\frac{dt}{t^{n+1}}.\\
&\lesssim r^{n/q}\|b\|_{{CBMO}^{p,\lambda}(\mathbb{R}^n)}\int_{2r}^{\infty}t^{n\lambda}\big(1+\ln\frac{t}{r}\big)\|f\|_{L^{q_1}(B(0,t))}\frac{dt}{t^{n/q_1+1}}.
\end{align*}
Combining all of the above estimates, we finish the proof of Lemma \ref{le3.3}.
\end{proof}

Now we are in a position to prove Theorems \ref{th3.1}-\ref{th3.4}.
\begin{proof}[Proof of Theorem \ref{th3.1}]
The method of the proof is standard, by Lemma \ref{le3.1} and a change of variables $t=s^{-\frac{q}{n}}$, we obtain that
\begin{align*}
\|\mathcal{G}_{\alpha}(f)\|_{\mathcal{B}^{q,\varphi_2}(\mathbb{R}^n)}&\lesssim\sup\limits_{r>0}\frac{1}{\varphi_2(r)}\frac{1}{|B(0,r)|^{1/q}} r^{n/q}\int_{2r}^{\infty}\|f\|_{L^q(B(0,t))}\frac{dt}{t^{n/q+1}}\\
&\lesssim\sup\limits_{r>0}\frac{1}{\varphi_2(r)}\int_{0}^{r^{-\frac{n}{q}}}\|f\|_{L^q(B(0,s^{-\frac{q}{n}}))}ds\\
&\lesssim\sup\limits_{r>0}\frac{1}{\varphi_2(r^{-\frac{q}{n}})}\int_{0}^{r}\|f\|_{L^q(B(0,s^{-\frac{q}{n}}))}ds\\
&=\sup\limits_{r>0}\frac{r}{\varphi_2(r^{-\frac{q}{n}})}\frac{1}{r}\int_{0}^{r}\|f\|_{L^q(B(0,s^{-\frac{q}{n}}))}ds.
\end{align*}
If we set
$$\omega(t)={\varphi_2(t^{-\frac{q}{n}})^{-1}t}, \quad \nu(t)={\varphi_1(t^{-\frac{q}{n}})^{-1}t},$$
since the pair $(\varphi_1,\varphi_2)$ satisfies the following condition
$$\int_{r}^{\infty}\frac{ess\inf_{t<\tau<\infty}\varphi_1(\tau){\tau}^{n/q}}{{t^{n/q+1}}}dt\lesssim\varphi_2(r),$$
it follows that
$$\sup\limits_{t>0}\frac{\omega(t)}{t}\int_{0}^{t}\frac{dr}{ess\sup_{0<s<t}\nu(s)}<\infty.$$
This leads to the following inequality (see \cite{C})
$$ess\sup_{t>0}\omega(t)\mathcal{H}g(t)\lesssim ess\sup_{t>0}\nu(t)g(t)$$
holds for all nonnegative and non-increasing functions $g$ on $(0,\infty)$, where $\mathcal{H}$ is the classical Hardy operator, that is,
$$\mathcal{H}g(t)=\frac{1}{t}\int_{0}^{t}g(r)dr.$$
Therefore, let $g(t)=\|f\|_{L^q(B(0,t^{-\frac{q}{n}}))}$, we have
\begin{align*}
\|\mathcal{G}_{\alpha}(f)\|_{\mathcal{B}^{q,\varphi_2}(\mathbb{R}^n)}
&\lesssim\sup\limits_{r>0}\frac{r}{\varphi_1(r^{-\frac{q}{n}})}\|f\|_{L^q(B(0,r^{-\frac{q}{n}}))}\leq\|f\|_{\mathcal{B}^{q,\varphi_1}(\mathbb{R}^n)}.
\end{align*}
The proof of Theorem \ref{th3.1} is completed.
\end{proof}
\begin{proof}[Proof of Theorem \ref{th3.2}] It is easy to see that the following fact
\begin{align*}
&\Big(g_{\tilde{\lambda},\alpha}^{*}(f)(x)\Big)^2=\int_{0}^{\infty}\int_{|x-y|<t}\Big(\frac{t}{t+|x-y|}\Big)^{n\tilde{\lambda}}\big(\mathcal{A}_{\alpha}f(t,y)\big)^2\frac{dydt}{t^{n+1}}\\
&\quad +\int_{0}^{\infty}\int_{|x-y|\geq t}\Big(\frac{t}{t+|x-y|}\Big)^{n\tilde{\lambda}}\big(\mathcal{A}_{\alpha}f(t,y)\big)^2\frac{dydt}{t^{n+1}}\\
&\leq\Big(\mathcal{G}_{\alpha}(f)(x)\Big)^2+\sum\limits_{j=1}^{\infty}\int_{0}^{\infty}\int_{2^{j-1}t\leq|x-y|<2^{j} t}\Big(\frac{t}{t+|x-y|}\Big)^{n\tilde{\lambda}}\big(\mathcal{A}_{\alpha}f(t,y)\big)^2\frac{dydt}{t^{n+1}}\\
&\lesssim\Big(\mathcal{G}_{\alpha}(f)(x)\Big)^2+\sum\limits_{j=1}^{\infty}\int_{0}^{\infty}\int_{2^{j-1}t\leq|x-y|<2^{j} t}2^{-jn\tilde{\lambda}}\big(\mathcal{A}_{\alpha}f(t,y)\big)^2\frac{dydt}{t^{n+1}}\\
&\leq\Big(\mathcal{G}_{\alpha}(f)(x)\Big)^2+\sum\limits_{j=1}^{\infty}2^{-jn\tilde{\lambda}}\int_{0}^{\infty}\int_{|x-y|<2^{j} t}\big(\mathcal{A}_{\alpha}f(t,y)\big)^2\frac{dydt}{t^{n+1}}\\
&\leq\Big(\mathcal{G}_{\alpha}(f)(x)\Big)^2+\sum\limits_{j=1}^{\infty}2^{-jn\tilde{\lambda}}\Big(\mathcal{G}_{\alpha,2^{j}}(f)(x)\Big)^2.
\end{align*}
Thus,
$$\|g_{\tilde{\lambda},\alpha}^{*}(f)\|_{\mathcal{B}^{q,\varphi_2}(\mathbb{R}^n)}\leq\|\mathcal{G}_{\alpha}(f)\|_{\mathcal{B}^{q,\varphi_2}(\mathbb{R}^n)}
+\sum\limits_{j=1}^{\infty}2^{\frac{-jn\tilde{\lambda}}{2}}\|\mathcal{G}_{\alpha,2^{j}}(f)\|_{\mathcal{B}^{q,\varphi_2}(\mathbb{R}^n)}.$$
Theorem \ref{th3.1} tells us the fact that
$$\|\mathcal{G}_{\alpha}(f)\|_{\mathcal{B}^{q,\varphi_2}(\mathbb{R}^n)}\lesssim\|f\|_{\mathcal{B}^{q,\varphi_1}(\mathbb{R}^n)}.$$
Therefore, to obtain the proof of Theorem \ref{th3.2}, it suffices to show that
$$\sum\limits_{j=1}^{\infty}2^{-jn\tilde{\lambda}}\|\mathcal{G}_{\alpha,2^{j}}(f)\|_{\mathcal{B}^{q,\varphi_2}(\mathbb{R}^n)}\lesssim\|f\|_{\mathcal{B}^{q,\varphi_1}(\mathbb{R}^n)},\quad\text{for $\tilde{\lambda}>3+\frac{2\alpha}{n}$}.$$
In fact, by Lemma \ref{le3.2} and using the similar to the proof of Theorem \ref{th3.1}, for any $\tilde{\lambda}>3+\frac{2\alpha}{n}$, we get that
\begin{align*}
\sum\limits_{j=1}^{\infty}2^{-jn\tilde{\lambda}}\|\mathcal{G}_{\alpha,2^{j}}(f)\|_{\mathcal{B}^{q,\varphi_2}(\mathbb{R}^n)}&\lesssim\sum
\limits_{j=1}^{\infty}2^{\frac{-jn\tilde{\lambda}}{2}}2^{j(3+\frac{2\alpha}{n})}\\
&\quad\times\sup\limits_{r>0}\frac{1}{\varphi_2(r^{-\frac{q}{n}})}\int_{0}^{r}\|f\|_{L^q(B(0,s^{-\frac{q}{n}}))}ds\\
&\lesssim\|f\|_{\mathcal{B}^{q,\varphi_1}(\mathbb{R}^n)}.
\end{align*}
The proof of Theorem \ref{th3.2} is finished.
\end{proof}

\begin{proof}[Proof of Theorem \ref{th3.3}]
By Lemma \ref{le3.3}, the proof of Theorem \ref{th3.3} is only a repetition of Theorem \ref{th3.1}. Therefore we omit the details here.
\end{proof}
\begin{proof}[Proof of Theorem \ref{th3.4}]
Simialr to the proof of Theorem \ref{th3.2}, together with Theorem \ref{th3.3}, we derive that
\begin{align*}
&\|[b,g_{\tilde{\lambda},\alpha}^{*}](f)\|_{\mathcal{B}^{q,\varphi_2}(\mathbb{R}^n)}\\
&\leq\|[b,\mathcal{G}_{\alpha}](f)\|_{\mathcal{B}^{q,\varphi_2}(\mathbb{R}^n)}
+\sum\limits_{j=1}^{\infty}2^{\frac{-jn\tilde{\lambda}}{2}}\|[b,\mathcal{G}_{\alpha,2^{j}}](f)\|_{\mathcal{B}^{q,\varphi_2}(\mathbb{R}^n)}\\
&\lesssim\|b\|_{{CBMO}^{p,\lambda}(\mathbb{R}^n)}\|f\|_{\mathcal{B}^{q_1,\varphi_1}(\mathbb{R}^n)}.
\end{align*}
The proof of Theorem \ref{th3.4} is completed.
\end{proof}

\end{document}